\documentclass[11pt]{article}

\usepackage{amsthm}
\usepackage{amsmath}
\usepackage{amsfonts}
\usepackage[all,cmtip]{xy}

\newtheorem{axiom}{Axiom}[section]
\newtheorem{theorem}{Theorem}[section]

\begin{document}

%%%%%%%%%%%%%%%%%%%%%%%%%%%%%%%%%%%%%%%%%%%%%%%%%%%%%%%%%%%%%%%%%%%%%%%%
%%%%%%%%%%%%%%%%%%%%%%%%%%%%%%%%%%%%%%%%%%%%%%%%%%%%%%%%%%%%%%%%%%%%%%%%
%%                                                                    %%
%%                             Top Matter                             %%
%%                                                                    %%
%%%%%%%%%%%%%%%%%%%%%%%%%%%%%%%%%%%%%%%%%%%%%%%%%%%%%%%%%%%%%%%%%%%%%%%%
%%%%%%%%%%%%%%%%%%%%%%%%%%%%%%%%%%%%%%%%%%%%%%%%%%%%%%%%%%%%%%%%%%%%%%%%
\begin{titlepage}

% Title
\vskip 1.5in
\begin{center}
{\bf\Large{Axiomatic TQFT, Axiomatic DQFT, and \\ Exotic 4-Manifolds}}\vskip 0.5cm 
{Kelly J. Davis}
\end{center}
\vskip 0.5in
\baselineskip 16pt

% Date
\date{June, 2011}

% Abstract
\begin{abstract}
In this article we prove that any unitary, axiomatic topological quantum
field theory in four-dimensions can not detect changes in the smooth 
structure of $M$, a simply connected, closed (compact without boundary), 
oriented smooth manifold. However, as Donaldson-Witten theory (a 
topological quantum field theory but not an axiomatic one) is able to
detect changes in the smooth structure of such an $M$, this seemingly
leads to a contradiction. This seeming contradiction is resolved by
introducing a new set of axioms for a ``differential quantum field 
theory'', which in truth only slightly modify the naturality and 
functoriality axioms of a topological quantum field theory, such
that these new axioms allow for a theory to detect changes in smooth
structure.
\end{abstract}

\end{titlepage}

\vfill
\eject 
\tableofcontents

%%%%%%%%%%%%%%%%%%%%%%%%%%%%%%%%%%%%%%%%%%%%%%%%%%%%%%%%%%%%%%%%%%%%%%%%
%%%%%%%%%%%%%%%%%%%%%%%%%%%%%%%%%%%%%%%%%%%%%%%%%%%%%%%%%%%%%%%%%%%%%%%%
%%                                                                    %%
%%                                Body                                %%
%%                                                                    %%
%%%%%%%%%%%%%%%%%%%%%%%%%%%%%%%%%%%%%%%%%%%%%%%%%%%%%%%%%%%%%%%%%%%%%%%%
%%%%%%%%%%%%%%%%%%%%%%%%%%%%%%%%%%%%%%%%%%%%%%%%%%%%%%%%%%%%%%%%%%%%%%%%
%%%%%%%%%%%%%%%%%%%%%%%%%%%%%%%%%%%%%%%%%%%%%%%%%%%%%%%%%%%%%%%%%%%%%%%%
%%%%%%%%%%%%%%%%%%%%%%%%%%%%%%%%%%%%%%%%%%%%%%%%%%%%%%%%%%%%%%%%%%%%%%%%
%%                                                                    %%
%%                             Introduction                           %%
%%                                                                    %%
%%%%%%%%%%%%%%%%%%%%%%%%%%%%%%%%%%%%%%%%%%%%%%%%%%%%%%%%%%%%%%%%%%%%%%%%
%%%%%%%%%%%%%%%%%%%%%%%%%%%%%%%%%%%%%%%%%%%%%%%%%%%%%%%%%%%%%%%%%%%%%%%%
\section{Introduction}
\label{Section:Introduction}

The fountainhead of modern topological quantum field theory can be
traced back to the work of Witten~\cite{Witten82}. There he proved
that the Morse inequalities~\cite{Milnor63} can be obtained through
the use of a certain supersymmetric version of quantum mechanics.
The next major milestone in the study of topological quantum field
theory was also authored by Witten~\cite{Witten88}. There he 
showed that the Donaldson polynomials~\cite{Donaldson90} can be 
interpreted as observables of a certain four-dimensional quantum
field theory. Subsequently, Witten~\cite{Witten89} authored 
\textit{Quantum Field Theory and the Jones Polynomial}, 
for which, in large part, he 
received the Fields Medal. There he proved that the Jones 
polynomial~\cite{Jones85} can be interpreted as an observable of
a certain three-dimensional quantum field theory. It was after the
publication of this paper that the flood gates opened and the
volume of papers dealing with topological quantum field theory
began to greatly increase.

With this increased volume of work on topological quantum field theory,
many mathematicians started to become more interested in the subject.
However, due to the methods used, in particular the mathematically
ill-defined path-integral~\cite{Ashtekar74}, many of the results 
were valid to a physicist's level of rigour, but not to a 
mathematician's. This soon changed when Atiyah~\cite{Atiyah88}, 
motivated by Witten~\cite{Witten89} and Segal~\cite{Segal88}, 
axiomatized the foundations of topological quantum field theory. This
axiomatization made it possible for mathematicians to obtain
rigorous results.

However, Atiyah's axiomatization is based on experiences from
topological quantum field theories in three or fewer dimensions.
The axiomatization is rarely used in four or more dimensions. 
Hence, there are a dearth of results using axiomatic topological 
quantum field theory in four or more dimensions, and the axioms
themselves may contain hidden ``biases'' that ``favor''
three or fewer dimensions. In particular, application of the
axiomatization to four dimensions~\cite{Thurston11} tends to
yield theories that are ``trivial'' in that they can not detect
changes in smooth structure\footnote{Thurston's mathoverflow
answer~\cite{Thurston11} and subsequent discussion were the
original motivation for this article.}.

In this article we will take a first step in to higher dimensions and
examine axiomatic topological quantum field theory in four 
dimensions. We prove, formalizing the difficulties expressed
by Thurston~\cite{Thurston11}, that in four dimensions any 
unitary, axiomatic topological quantum field theory can not 
detect changes in the smooth structure of $M$, a simply connected, 
closed (compact without boundary), oriented smooth 
four-manifold. This motivates us to slightly modify the axioms of a
topological quantum field theory so that it is possible for an axiomatic 
topological quantum field theory to detect changes in the
smooth structure of such an $M$. Thus, these modified axioms 
could more accurately be dubbed axioms of a differential
quantum field theory.

%%%%%%%%%%%%%%%%%%%%%%%%%%%%%%%%%%%%%%%%%%%%%%%%%%%%%%%%%%%%%%%%%%%%%%%%
%%%%%%%%%%%%%%%%%%%%%%%%%%%%%%%%%%%%%%%%%%%%%%%%%%%%%%%%%%%%%%%%%%%%%%%%
%%                                                                    %%
%%                           Axiomatic TQFT                           %%
%%                                                                    %%
%%%%%%%%%%%%%%%%%%%%%%%%%%%%%%%%%%%%%%%%%%%%%%%%%%%%%%%%%%%%%%%%%%%%%%%%
%%%%%%%%%%%%%%%%%%%%%%%%%%%%%%%%%%%%%%%%%%%%%%%%%%%%%%%%%%%%%%%%%%%%%%%%
\section{Axiomatic TQFT}
\label{Section:AxiomaticTQFT}

In his ground--breaking work Witten~\cite{Witten88} introduced an 
``informal'' definition of a topological quantum field 
theory, a quantum field theory on a smooth manifold $M$ that is 
independent of the metric placed on $M$. Atiyah~\cite{Atiyah88},
motivated by Witten's informal definition and Segal's~\cite{Segal88}
axiomatization of two-dimensional conformal field theory, then axiomatized 
topological quantum field theory. Over the years several 
authors have explored and refined Atiyah's  
axiomatization, see~\cite{Quinn:1991kq} and~\cite{Turaev:1994xb}, 
resulting in the current formulation~\cite{Blanchet2006232}, which we 
describe below.

%%%%%%%%%%%%%%%%%%%%%%%%%%%%%%%%%%%%%%%%%%%%%%%%%%%%%%%%%%%%%%%%%%%%%%%%
%%                                                                    %%
%%                           Axiomatic TQFT                           %%
%%                                                                    %%
%%%%%%%%%%%%%%%%%%%%%%%%%%%%%%%%%%%%%%%%%%%%%%%%%%%%%%%%%%%%%%%%%%%%%%%%
\subsection{Axiomatic TQFT}
\label{SubSection:AxiomaticTQFT}

An $(n + 1)$-dimensional topological quantum field theory, from now on
abbreviated TQFT, over a field $\mathbb{F}$ assigns to every closed,
oriented $n$-dimensional smooth manifold $X$ a finite dimensional 
vector space $\mathcal{H}(X)$ over $\mathbb{F}$ and assigns to every 
$(n+1)$-dimensional cobordism $W$ from $X_-$ to $X_+$
an $\mathbb{F}$ linear map,
\begin{equation}
Z(W,X_-,X_+): 
\mathcal{H}(X_-) \rightarrow \mathcal{H}(X_+).
\end{equation}
Recall that given two closed, oriented $n$-dimensional smooth manifolds
$X_\pm$ a \textit{cobordisim} from $X_-$ to $X_+$ is a compact, 
oriented $(n+1)$-dimensional smooth manifold $W$ such that $\partial 
W = X_- \amalg  X_+$, where $\partial W$ is the boundary of $W$
and $\amalg$ denotes disjoint union. The assignments $\mathcal{H}(X)$
and $Z(W,X_-,X_+)$ must satisfy the following axioms.

%%%%%%%%%%%%%%%%%%%%%%%%%%%%%%%%%%%%%%%%%%%%%%%%%%%%%%%%%%%%%%%%%%%%%%%%
%%                             Naturality                             %%
%%%%%%%%%%%%%%%%%%%%%%%%%%%%%%%%%%%%%%%%%%%%%%%%%%%%%%%%%%%%%%%%%%%%%%%%
\subsubsection{Naturality}
\label{SubSubSection:Naturality}

\begin{axiom}[Naturality]
\label{Axiom:Naturality}
Any orientation--preserving diffeomorphism of closed, oriented
$n$-dimensional smooth manifolds $f:X \rightarrow X'$ induces an 
isomorphism\footnote{Note, we use $f$ to denote the 
orientation--preserving diffeomorphism and the isomorphism. Context
should prevent any confusion in this regard.} $f: \mathcal{H}(X) 
\rightarrow \mathcal{H}(X')$. For an orientation--preserving diffeomorphism 
$g$ from the cobordism $(W,X_-,X_+)$ to the cobordism $(W',X_-',X_+')$, the 
following diagram is commutative.
\[
\xymatrixcolsep{5pc}
\xymatrix{
\mathcal{H}(X_-) \ar[d]_{Z(W)} \ar[r]^{g_{|_{X_-}}} &\mathcal{H}(X_-')\ar[d]^{Z(W')}\\
\mathcal{H}(X_+) \ar[r]^{g_{|_{X_+}}} &\mathcal{H}(X_+')}
\]
Note, $Z(W)$ is shorthand for $Z(W,X_-,X_+)$ and
$Z(W')$ is shorthand for $Z(W',X_-',X_+')$.
\end{axiom}

%%%%%%%%%%%%%%%%%%%%%%%%%%%%%%%%%%%%%%%%%%%%%%%%%%%%%%%%%%%%%%%%%%%%%%%%
%%                            Functoriality                           %%
%%%%%%%%%%%%%%%%%%%%%%%%%%%%%%%%%%%%%%%%%%%%%%%%%%%%%%%%%%%%%%%%%%%%%%%%
\subsubsection{Functoriality}
\label{SubSubSection:Functoriality}

\begin{axiom}[Functoriality]
\label{Axiom:Functoriality}
If a cobordism $(W,X_-,X_+)$ is obtained by gluing\footnote{The formal
definition of \textit{gluing} is given in Chapter VI Section 5 of 
Kosinski~\cite{Kosinski93}.} two cobordisms $(M,X_-,X)$ and 
$(M',X',X_+)$ using an orientation-preserving diffeomorphism 
$f: X \rightarrow X'$, then the following diagram is commutative.
\[
\xymatrixcolsep{5pc}
\xymatrix{
\mathcal{H}(X_-) \ar[d]_{Z(M)} \ar[r]^{Z(W)} &\mathcal{H}(X_+)\\
\mathcal{H}(X) \ar[r]^{f} &\mathcal{H}(X')\ar[u]_{Z(M')}}
\]
\end{axiom}

%%%%%%%%%%%%%%%%%%%%%%%%%%%%%%%%%%%%%%%%%%%%%%%%%%%%%%%%%%%%%%%%%%%%%%%%
%%                            Normalization                           %%
%%%%%%%%%%%%%%%%%%%%%%%%%%%%%%%%%%%%%%%%%%%%%%%%%%%%%%%%%%%%%%%%%%%%%%%%
\subsubsection{Normalization}
\label{SubSubSection:Normalization}

\begin{axiom}[Normalization]
\label{Axiom:Normalization}
For any closed, oriented $n$-dimensional smooth manifold $X$, the
$\mathbb{F}$ linear map
\begin{equation*}
Z(X \times [0,1]): 
\mathcal{H}(X) \rightarrow \mathcal{H}(X)
\end{equation*}
is the identity.
\end{axiom}

%%%%%%%%%%%%%%%%%%%%%%%%%%%%%%%%%%%%%%%%%%%%%%%%%%%%%%%%%%%%%%%%%%%%%%%%
%%                          Multiplicativity                          %%
%%%%%%%%%%%%%%%%%%%%%%%%%%%%%%%%%%%%%%%%%%%%%%%%%%%%%%%%%%%%%%%%%%%%%%%%
\subsubsection{Multiplicativity}
\label{SubSubSection:Multiplicativity}

\begin{axiom}[Multiplicativity]
\label{Axiom:Multiplicativity}
There are functorial isomorphisms
\begin{equation*}
\mathcal{H}(X \amalg Y) \longrightarrow \mathcal{H}(X) \otimes \mathcal{H}(Y)
\end{equation*}
and
\begin{equation*}
\mathcal{H}(\emptyset) \longrightarrow \mathbb{F}
\end{equation*}
such that the diagrams
\[
\xymatrix{
\mathcal{H}((X_1 \amalg X_2) \amalg X_3) \ar[d] \ar[r] &(\mathcal{H}(X_1) \otimes \mathcal{H}(X_2)) \otimes \mathcal{H}(X_3)\ar[d]\\
\mathcal{H}(X_1 \amalg (X_2 \amalg X_3)) \ar[r] &\mathcal{H}(X_1) \otimes (\mathcal{H}(X_2) \otimes \mathcal{H}(X_3))}
\]
and
\[
\xymatrix{
\mathcal{H}(X \amalg \emptyset) \ar[d] \ar[r] &\mathcal{H}(X) \otimes \mathbb{F}\ar[d] \\
\mathcal{H}(X) \ar[r]^{id} &\mathcal{H}(X)}
\]
commute. Note, the vertical maps are induced by the obvious
diffeomorphisms and the standard vector space isomorphisms. 
\end{axiom}

%%%%%%%%%%%%%%%%%%%%%%%%%%%%%%%%%%%%%%%%%%%%%%%%%%%%%%%%%%%%%%%%%%%%%%%%
%%                              Symmetry                              %%
%%%%%%%%%%%%%%%%%%%%%%%%%%%%%%%%%%%%%%%%%%%%%%%%%%%%%%%%%%%%%%%%%%%%%%%%
\subsubsection{Symmetry}
\label{SubSubSection:Symmetry}

\begin{axiom}[Symmetry]
\label{Axiom:Symmetry}
The isomorphism
\begin{equation*}
\mathcal{H}(X \amalg Y) \longrightarrow \mathcal{H}(Y \amalg X)
\end{equation*}
induced by the obvious diffeomorphism corresponds to the standard
isomorphism of vector spaces
\begin{equation*}
\mathcal{H}(X) \otimes \mathcal{H}(Y) \longrightarrow \mathcal{H}(Y) \otimes \mathcal{H}(X).
\end{equation*}
\end{axiom}

%%%%%%%%%%%%%%%%%%%%%%%%%%%%%%%%%%%%%%%%%%%%%%%%%%%%%%%%%%%%%%%%%%%%%%%%
%%                                                                    %%
%%                               Remarks                              %%
%%                                                                    %%
%%%%%%%%%%%%%%%%%%%%%%%%%%%%%%%%%%%%%%%%%%%%%%%%%%%%%%%%%%%%%%%%%%%%%%%%
\subsection{Remarks}
\label{SubSection:Remarks}

Before continuing on with the remainder of this article, there are
a few points of note that easily follow from the above axioms and
that we will have need of later. 

First, an axiomatic TQFT defines invariants for closed, oriented 
$(n+1)$-dimensional smooth manifolds. In more detail, a closed, 
oriented $(n+1)$-dimensional smooth manifold $W$ can be thought 
of as a cobordism from  $\emptyset$ to $\emptyset$. Thus, $Z(W) 
\in Hom_{\mathbb{F}}(\mathbb{F}, \mathbb{F}) = \mathbb{F}$, and 
$Z(W) \in \mathbb{F}$ is simply a numerical invariant of $W$.

Second, any compact, oriented $(n+1)$-dimensional smooth manifold
$W$ with boundary can be thought of as a cobordism from  
$\emptyset$ to $\partial W$. Thus, $Z(W) \in Hom_{\mathbb{F}}
(\mathbb{F}, \mathcal{H}(\partial W)) = \mathcal{H}(\partial W)$. 
So, $Z(W)$ in this case is simply a vector in $\mathcal{H}
(\partial W)$. This vector $Z(W)$ is called the 
\textit{vacuum vector} of $W$ and we will find it of great use in 
what follows.

Finally, for a closed, oriented $n$-dimensional smooth manifold
$X$ the manifold $X \times [0,1]$ can be considered as a cobordism
from  $\overline{X} \amalg X$ to $\emptyset$, where $\overline{X}$
is $X$ with its orientation reversed. Hence, $Z(X \times [0,1])$ 
can be viewed as an $\mathbb{F}$ linear map
\begin{equation}
Z(X \times [0,1]): 
\mathcal{H}(\overline{X}) \otimes \mathcal{H}(X) \rightarrow \mathbb{F}.
\end{equation}
This gives a functorial isomorphism of $\mathcal{H}(\overline{X}) =
\mathcal{H}(X)^* = Hom_{\mathbb{F}}(\mathcal{H}(X), \mathbb{F})$. 

Thus, if a closed, oriented $(n+1)$-dimensional smooth manifold $W$ 
is obtained by gluing $M$ to $M'$, where $\partial M = 
\overline{\partial M'}$, then Axiom~\ref{Axiom:Functoriality}, the functoriality axiom, implies $Z(W) = \langle Z(M') | Z(M) \rangle 
\in \mathbb{F}$, where $Z(M)$ and $Z(M')$ are viewed as vacuum 
vectors and $\langle Z(M') | Z(M) \rangle$ is defined as the value of 
$Z(M') \in \mathcal{H}(\partial M)^*$ acting on $Z(M) \in \mathcal{H}
(\partial M)$.

%%%%%%%%%%%%%%%%%%%%%%%%%%%%%%%%%%%%%%%%%%%%%%%%%%%%%%%%%%%%%%%%%%%%%%%%
%%                                                                    %%
%%                             Unitarity                              %%
%%                                                                    %%
%%%%%%%%%%%%%%%%%%%%%%%%%%%%%%%%%%%%%%%%%%%%%%%%%%%%%%%%%%%%%%%%%%%%%%%%
\subsection{Unitarity}
\label{SubSection:Unitarity}

An additional axiom that is sometimes used in conjunction with the
above set of standard axioms is that of unitarity.

\begin{axiom}[Unitarity]
\label{Axiom:Unitarity}
For any compact, oriented $(n+1)$-dimensional smooth manifold $W$
with non-zero $Z(W) \in \mathcal{H}(\partial W)$ the element 
$Z(\overline{W} \cup_{id} W) = \langle Z(\overline{W}) | Z(W) \rangle 
\in \mathbb{F}$ is not zero.
\end{axiom}

Unitarity  is sometimes, but not always, taken as an axiom of
TQFT. However, all ``physical'' theories, for example the 
standard model~\cite{Peskin95} and general relativity~\cite{Hawking05},
are unitary. Thus, we will assume that any axiomatic TQFT that we deal
with obeys the unitarity axiom.

%%%%%%%%%%%%%%%%%%%%%%%%%%%%%%%%%%%%%%%%%%%%%%%%%%%%%%%%%%%%%%%%%%%%%%%%
%%%%%%%%%%%%%%%%%%%%%%%%%%%%%%%%%%%%%%%%%%%%%%%%%%%%%%%%%%%%%%%%%%%%%%%%
%%                                                                    %%
%%               Akbulut Corks and Exotic Four-Manifolds              %%
%%                                                                    %%
%%%%%%%%%%%%%%%%%%%%%%%%%%%%%%%%%%%%%%%%%%%%%%%%%%%%%%%%%%%%%%%%%%%%%%%%
%%%%%%%%%%%%%%%%%%%%%%%%%%%%%%%%%%%%%%%%%%%%%%%%%%%%%%%%%%%%%%%%%%%%%%%%
\section{Akbulut Corks and Exotic Four-Manifolds}
\label{Section:AkbulutCorksAndExoticFourManifolds}

The wellspring of many an idea related to exotic four-manifolds can
be traced back to the work of Akbulut~\cite{Akbulut88}. In this
foundational work Akbulut found that for a certain smooth four-manifold 
$M$ one can make an exotic copy $M'$ of $M$, a manifold homeomorphic 
but not diffeomorphic to $M$, by cutting out and regluing $A_C$, a
certain four-dimensional smooth submanifold of $M$, by an involution 
of its boundary $\partial A_C$. This smooth four-manifold $A_C$ later 
became known as  an Akbulut cork.

This means of generating exotic four-manifolds was later 
generalized in a preprint of Curtis and Hsiang. The proofs in 
this preprint were then simplified and extended through the work
of Curtis, Freedman, Hsiang, and Strong~\cite{Curtis96}, 
Matveyev~\cite{Matveyev95}, Bi\u{z}aca, and Kirby~\cite{Kirby97}.

In this section, to place these developments in the proper context, 
we will review the theorems that built up to the discovery of Akbulut 
corks, Smale's h-cobordisim theorem~\cite{Smale62} and Freedman's 
h-cobordisim theorem~\cite{Freedman82}, as well as reviewing the 
theorems presented in the above series of papers. These theorems will
be presented without proofs. The interested reader can refer to 
original works and/or to Chapter 9 of Gompf and 
Stipsicz~\cite{Gompf99} where most of this material is covered.

%%%%%%%%%%%%%%%%%%%%%%%%%%%%%%%%%%%%%%%%%%%%%%%%%%%%%%%%%%%%%%%%%%%%%%%%
%%                                                                    %%
%%                    Smale's h-Cobordisim Theorem                    %%
%%                                                                    %%
%%%%%%%%%%%%%%%%%%%%%%%%%%%%%%%%%%%%%%%%%%%%%%%%%%%%%%%%%%%%%%%%%%%%%%%%
\subsection{Smale's h-Cobordisim Theorem}
\label{SubSection:SmalesHCobordisimTheorem}

Classification of four-dimensional smooth manifolds up to 
diffeomorphism can best be understood, strangely enough, by 
looking first at the classification of smooth manifolds up to
diffeomorphism in greater than four dimensions. Looking at the
results in higher dimensions serves to put the results 
in four dimensions in to the proper context.

The key result used to classify manifolds up to diffeomorphism
in greater than four dimensions is Smale's h-cobordisim 
theorem~\cite{Smale62}. This theorem establishes a criteria 
through which one can determine if two simply connected, closed, 
oriented smooth $n$-manifolds, where $n > 4$, are diffeomorphic.
It is this theorem which we will now review.

However, before presenting Smale's h-cobordisim theorem, we must
introduce a definition~\cite{Gompf99}. Two simply connected smooth
manifolds $X_-$ and $X_+$ are \textit{h-cobordant} if there exists 
a cobordisim $W$ from $X_-$ to $X_+$ such that the inclusions 
$i_\pm: X_\pm \hookrightarrow W$ are homotopy equivalences. Given 
this definition we can now state Smale's h-cobordisim 
theorem~\cite{Gompf99}.

\begin{theorem}[Smale's h-Cobordisim Theorem]
\label{Theorem:SmaleHCobordisimTheorem}
If $W$ is an h-cobordisim between the $n$-dimensional smooth 
manifolds $X_-$ and $X_+$, where $n > 4$, then $W$ is diffeomorphic
to $X_- \times [0,1] $. In particular $X_-$ is diffeomorphic to 
$X_+$. 
\end{theorem}

With this one can see that if two $n$-dimensional smooth manifolds
are h-cobordant and $n > 4$, then these two manifolds are 
diffeomorphic. In practice this often simplifies the process of
determining if two manifolds are diffeomorphic, as proving two 
manifolds are h-cobordant is often easier than directly proving 
they are diffeomorphic.	 

This theorem can be used to classify smooth manifolds up to 
diffeomorphism in more than four dimensions. However, as we 
will see, this result fails to be true in four dimensions,
where a strictly ``weaker'' result holds. This ``weaker'' 
result is the subject of Freedman's h-cobordisim theorem
to which we now turn.

%%%%%%%%%%%%%%%%%%%%%%%%%%%%%%%%%%%%%%%%%%%%%%%%%%%%%%%%%%%%%%%%%%%%%%%%
%%                                                                    %%
%%                   Freedman's h-Cobordisim Theorem                  %%
%%                                                                    %%
%%%%%%%%%%%%%%%%%%%%%%%%%%%%%%%%%%%%%%%%%%%%%%%%%%%%%%%%%%%%%%%%%%%%%%%%
\subsection{Freedman's h-Cobordisim Theorem}
\label{SubSection:FreedmansHCobordisimTheorem}

One may hope that the techniques used to prove Smale's h-cobordisim 
theorem could be generalized to accommodate the case $n = 4$. However,
this is not possible\footnote{The main problem is that ``Whitney's
Trick'', which works in more than four dimensions, fails in four 
dimensions~\cite{Gompf99}.}.  The best one can do in four dimensions
is Freedman's h-cobordisim theorem~\cite{Gompf99}.

\begin{theorem}[Freedman's h-Cobordisim Theorem]
\label{Theorem:FreedmanHCobordisimTheorem}
If $W$ is an h-cobordisim between the four-dimensional smooth 
manifolds $X_-$ and $X_+$, then $W$ is homeomorphic to 
$X_- \times [0,1] $. In particular $X_-$ is homeomorphic to 
$X_+$. 
\end{theorem}

Thus, if two four-dimensional smooth manifolds are 
h-cobordant, then these two manifolds are homeomorphic. In 
four-dimensions this result can not be improved upon. In other words,
there exist four-dimensional smooth manifolds $X_-$ and $X_+$
that are h-cobordant and \textit{not} diffeomorphic~\cite{Gompf99}. 
As they are h-cobordant, Freedman's h-cobordisim theorem implies they
are homeomorphic. But, as they are not diffeomorphic, $X_+$ is an
exotic version of $X_-$, a manifold homeomorphic but not diffeomorphic
to $X_-$. In fact, the original results of Akbulut~\cite{Akbulut88} 
provide such a pair.

As it is a result we will require later, we pause here to note that
one can strengthen Freedman's h-cobordisim theorem in the following 
manner~\cite{Gompf99}.

\begin{theorem}[Strengthened Freedman's h-Cobordisim Theorem]
\label{Theorem:StrengthenedFreedmanHCobordisimTheorem}
Two simply connected, closed, oriented, four-dimensional smooth 
manifolds $X_-$ and $X_+$ are homeomorphic if and only if
they are h-cobordant. 
\end{theorem}

%%%%%%%%%%%%%%%%%%%%%%%%%%%%%%%%%%%%%%%%%%%%%%%%%%%%%%%%%%%%%%%%%%%%%%%%
%%                                                                    %%
%%                            Akbulut Corks                           %%
%%                                                                    %%
%%%%%%%%%%%%%%%%%%%%%%%%%%%%%%%%%%%%%%%%%%%%%%%%%%%%%%%%%%%%%%%%%%%%%%%%
\subsection{Akbulut Corks}
\label{SubSection:AkbulutCorks}

The results of Akbulut~\cite{Akbulut88}, along with Smale's and Freedman's 
h-cobordisim theorems, lead one to conjecture that it might be possible 
to ``excise'' a submanifold $A$ from $W$, a five-dimensional 
h-cobordisim from $X_-$ to $X_+$, such that the remainder $W - int(A)$ 
is diffeomorphic to $(X_- - int(A)) \times [0,1]$. Thus,
all of the ``strangeness'' that occurs in four dimensions would be 
contained in $A$, and $W - int(A)$ would be ``trivial''. This 
conjecture, and in fact much more, is true, as was found by Curtis, 
Freedman, Hsiang, and Strong~\cite{Curtis96}, 
Matveyev~\cite{Matveyev95}, Bi\u{z}aca, and Kirby~\cite{Kirby97}.

The formal summary of the flurry of work contained in the above articles
is given by the following theorem~\cite{Kirby97}.

\begin{theorem}[Pr\'ecis of Akbulut Corks]
\label{Theorem:PrecisOfAkbulutCorks}
If $W$ is a five-dimensional h-cobordism between two 
smooth four-manifolds $X_-$ and $X_+$, then there exists a 
five-dimensional h-cobordism $A \subset W$ from 
the smooth four-manifold $A_- \subset X_-$ to the
smooth four-manifold $A_+ \subset X_+$ with the 
following properties:
\begin{description}
\item[(1)] $A_-$, and hence $A$ and $A_+$, is contractible.
\item[(2)] $W - int(A)$ is diffeomorphic to $(X_- - int(A_-)) \times [0,1]$.
\item[(3)] $W - A$, and hence $X_- - A_-$ and $X_+ - A_+$, is simply connected.
\item[(4)] $A$ is diffeomorphic to $D^5$, the standard five-dimensional disk with boundary.
\item[(5)] $A_- \times [0,1]$ and $A_+ \times [0,1]$ are diffeomorphic 
to $D^5$.
\item[(6)] $A_-$ is diffeomorphic to $A_+$ by a diffeomorphism which, 
when restricted to $\partial A_- = \partial A_+$, is an involution.
\end{description}
\end{theorem}

The manifolds $A_\pm$ identified above are Akbulut corks and are
a generalization of the manifolds first discovered by 
Akbulut~\cite{Akbulut88} in his foundational work.

Given $W$, $X_\pm$, and $A_\pm$ as appear in the previous theorem,
one can easily prove the following results. As a result of $(2)$,
$X_- - int(A_-)$ is diffeomorphic to $X_+ - int(A_+)$. The definitions
of $X_\pm$ and $A_\pm$ imply $X_\pm = (X_\pm - int(A_\pm)) \cup_{id}
A_\pm$. Thus, as a result of $(6)$, $X_- = (X_- - int(A_-)) \cup_{id}
A_-$ and $X_+ = (X_- - int(A_-)) \cup_{I} A_-$, where $I$ is
the involution of $\partial A_-$ from $(6)$ and all 
equivalences are up to diffeomorphism.

Now, assume one has a simply connected, closed, oriented, smooth 
four-manifold $M$ along with $M'$, a manifold homeomorphic but not 
diffeomorphic to $M$. (In other words, $M'$ is an exotic version of
$M$.) As a result of the strengthened version of Freedman's 
h-cobordisim theorem, $M$ is h-cobordant to $M'$. Thus, as a result
of the argument in the previous paragraph, there exists an Akbulut 
cork $A_C \subset M$ such that $M = (M - int(A_C)) \cup_{id} A_C$
and $M' = (M - int(A_C)) \cup_{I} A_C$, where $I$ is the 
involution of $\partial A_C$ given in $(6)$.

%%%%%%%%%%%%%%%%%%%%%%%%%%%%%%%%%%%%%%%%%%%%%%%%%%%%%%%%%%%%%%%%%%%%%%%%
%%%%%%%%%%%%%%%%%%%%%%%%%%%%%%%%%%%%%%%%%%%%%%%%%%%%%%%%%%%%%%%%%%%%%%%%
%%                                                                    %%
%%                Axiomatic TQFT and Exotic 4-Manifolds               %%
%%                                                                    %%
%%%%%%%%%%%%%%%%%%%%%%%%%%%%%%%%%%%%%%%%%%%%%%%%%%%%%%%%%%%%%%%%%%%%%%%%
%%%%%%%%%%%%%%%%%%%%%%%%%%%%%%%%%%%%%%%%%%%%%%%%%%%%%%%%%%%%%%%%%%%%%%%%
\section{Axiomatic TQFT and Exotic 4-Manifolds}
\label{Section:AxiomaticTQFTAndExotic4Manifolds}

This section will be dedicated to proving our main theorem.

\begin{theorem}
\label{Theorem:MainTheorem}
In four-dimensions any unitary, axiomatic topological quantum field 
theory can not detect changes in the smooth structure of $M$, a 
simply connected, closed (compact without boundary), oriented 
smooth four-manifold.
\end{theorem}

\begin{proof}
Assume there exists a smooth manifold $M'$ homeomorphic but not
diffeomorphic to $M$, in other words $M'$ has a different smooth
structure than $M$. We will prove that $Z(M) = Z(M')$ for any
unitary, axiomatic topological quantum field theory.

As $M$ and $M'$ are homeomorphic, 
Theorem~\ref{Theorem:StrengthenedFreedmanHCobordisimTheorem},
the strengthened Freedman's h-cobordisim theorem, implies that
there exists an h-cobordisim $W$ from $M$ to $M'$.

As there exists an h-cobordisim $W$ from $M$ to $M'$, 
Theorem~\ref{Theorem:PrecisOfAkbulutCorks} implies that there exists 
an Akbulut cork $A_C \subset M$ such that 
\begin{equation*}
M = (M - int(A_C))\cup_{id} A_C
\end{equation*}
and
\begin{equation*}
M' = (M - int(A_C)) \cup_{I} A_C,
\end{equation*}
where $I$ is the involution of $\partial A_C$ given in part
(6) of Theorem~\ref{Theorem:PrecisOfAkbulutCorks}.

As $M = (M - int(A_C)) \cup_{id} A_C$, the results of
Section~\ref{SubSection:Remarks} imply the equality 
\begin{equation*}
Z(M) = \langle Z(M - int(A_C)) | Z(A_C) \rangle
\end{equation*}
Similarly, as $M' = (M - int(A_C)) \cup_{I} A_C$, the results of
Section~\ref{SubSection:Remarks} along with 
Axiom~\ref{Axiom:Functoriality}, the functoriality axiom, imply 
\begin{equation*}
Z(M') =  \langle Z(M - int(A_C)) | I(Z(A_C)) \rangle,
\end{equation*}
where $I$ is the isomorphism of $\mathcal{H}(\partial A_C)$ induced
by the involution $I$ of $\partial A_C$. Thus, to prove $Z(M) = Z(M')$ 
we only need to prove $Z(A_C) = I(Z(A_C))$, or, equivalently, we need 
to prove $Z(A_C) - I(Z(A_C)) = 0$.

If $Z(A_C) - I(Z(A_C)) = 0$, then we are done. So, we can thus safely 
assume that $Z(A_C) - I(Z(A_C)) \ne 0$. Hence, 
Axiom~\ref{Axiom:Unitarity}, the unitarity axiom, implies that if 
the product $\langle Z(\overline{A_C}) - I(Z(\overline{A_C})) | Z(A_C) - I(Z(A_C)) 
\rangle = 0$, then $Z(A_C) - I(Z(A_C)) = 0$. So, if we can prove that
$\langle Z(\overline{A_C}) - I(Z(\overline{A_C})) | Z(A_C) - I(Z(A_C)) 
\rangle = 0$, we are done.

Now, using linearity along with our various definitions we have
\begin{align*}
\lefteqn{\langle Z(\overline{A_C}) - I(Z(\overline{A_C})) | Z(A_C) - I(Z(A_C)) 
\rangle} \\
&= \langle Z(\overline{A_C}) | Z(A_C) \rangle -
\langle Z(\overline{A_C}) | I(Z(A_C)) \rangle - \\
&\qquad \qquad \qquad {} \langle I(Z(\overline{A_C})) | Z(A_C) \rangle +
\langle I(Z(\overline{A_C})) | I(Z(A_C)) \rangle \\
&= Z(\overline{A_C} \cup_{id} A_C) -
Z(\overline{A_C} \cup_I A_C) -
Z(\overline{A_C} \cup_I A_C) +
Z(\overline{A_C} \cup_{I^2} A_C) \\
&= Z(\overline{A_C} \cup_{id} A_C) -
Z(\overline{A_C} \cup_I A_C) -
Z(\overline{A_C} \cup_I A_C) +
Z(\overline{A_C} \cup_{id} A_C) \\
&= 2 ( Z(\overline{A_C} \cup_{id} A_C) -
Z(\overline{A_C} \cup_I A_C) ),
\end{align*}
where in the second to last line we have used the fact that $I$ is
an involution and thus $I^2 = id$. As a result of the previous 
computation, we find that our desired conclusion follows if we can
prove $Z(\overline{A_C} \cup_{id} A_C) - Z(\overline{A_C} \cup_I A_C) =
0$.

Now, part (5) of Theorem~\ref{Theorem:PrecisOfAkbulutCorks}, pr\'ecis 
of Akbulut corks, implies $A_C \times [0,1]$ is diffeomorphic to
$D^5$, the standard five-dimensional disk with boundary. As
$\partial (A_C \times [0,1]) = \overline{A_C} \cup_{id} A_C$,
this implies $\overline{A_C} \cup_{id} A_C = S^4$, where $S^4$
is the standard four-dimensional sphere. Thus, $Z(\overline{A_C} 
\cup_{id} A_C) = Z(S^4)$.

Part (4) of Theorem~\ref{Theorem:PrecisOfAkbulutCorks}, pr\'ecis 
of Akbulut corks, implies $A$, of 
Theorem~\ref{Theorem:PrecisOfAkbulutCorks}, is diffeomorphic to
$D^5$. As $\partial A = \overline{A_C} \cup_I A_C$ in our case, 
this implies $\overline{A_C} \cup_I A_C = S^4$. Thus, 
$Z(\overline{A_C} \cup_I A_C) = Z(S^4)$.

Collecting the results of the last two paragraphs,
\begin{equation*}
Z(\overline{A_C} \cup_{id} A_C) - Z(\overline{A_C} \cup_I A_C) 
= Z(S^4) - Z(S^4)
= 0.
\end{equation*}
So, tracing all our previous steps, we have proven $Z(M) = Z(M')$.
\end{proof}

%%%%%%%%%%%%%%%%%%%%%%%%%%%%%%%%%%%%%%%%%%%%%%%%%%%%%%%%%%%%%%%%%%%%%%%%
%%%%%%%%%%%%%%%%%%%%%%%%%%%%%%%%%%%%%%%%%%%%%%%%%%%%%%%%%%%%%%%%%%%%%%%%
%%                                                                    %%
%%                               Remarks                              %%
%%                                                                    %%
%%%%%%%%%%%%%%%%%%%%%%%%%%%%%%%%%%%%%%%%%%%%%%%%%%%%%%%%%%%%%%%%%%%%%%%%
%%%%%%%%%%%%%%%%%%%%%%%%%%%%%%%%%%%%%%%%%%%%%%%%%%%%%%%%%%%%%%%%%%%%%%%%
\section{Remarks}
\label{Section:Remarks}

The results of Theorem~\ref{Theorem:MainTheorem} seem, somehow, 
unsatisfying. It is well known that Donaldson-Witten 
theory~\cite{Witten88} is a TQFT, in Witten's informal sense,
that is able to detect changes in the smooth structure of $M$, a
simply connected, closed, oriented smooth four-manifold. So,
it comes as somewhat of a surprise that any unitary, axiomatic TQFT
can not detect changes in the smooth structure of such an $M$.
It feels as if axiomatic TQFT is lacking something that is
present in Donaldson-Witten theory, and indeed this is the
case. However, the modifications that one must make to axiomatic
TQFT in order to allow it to detect changes in smooth structure
are relatively easy to spot upon thinking a bit about what is
happening in the scenario above.

Axiomatic TQFT in four-dimensions is, rather unsurprisingly,
a four-dimensional theory. So, in particular, all of its
symmetries should arise from symmetries that appear naturally
in four-dimensions. For example, Axiom~\ref{Axiom:Naturality}, 
the naturality axiom, implies that any axiomatic TQFT in
four-dimensions is invariant with respect to four-dimensional
diffeomorphisms. This makes sense. This is a purely four-dimensional
symmetry that arises in a purely four-dimensional theory.
However, by contrast, Axiom~\ref{Axiom:Naturality} also
states that any orientation--preserving diffeomorphism of closed, 
oriented, three-dimensional smooth manifolds $f:X \rightarrow X'$ 
induces an isomorphism $f: \mathcal{H}(X) \rightarrow \mathcal{H}(X')$.
At first glance this seems harmless, but, in fact, it is not.

The involution $I$ of $\partial A_C$ from part (6) of 
Theorem~\ref{Theorem:PrecisOfAkbulutCorks}, pr\'ecis of Akbulut corks,
is a diffeomorphism of $\partial A_C$ that does not arise from
a diffeomorphism of $A_C$. In other words one can not continue
$I$ over $A_C$ as a diffeomorphism. The best one can do is to
continue $I$ over $A_C$ as a homeomorphism\footnote{The easiest
way to see this is to note that if one could continue $I$ over
$A_C$ as a diffeomorphism, then one could prove Smale's
h-cobordism in four-dimensions, a result known to be false.}.
So, the assertion in Axiom~\ref{Axiom:Naturality} that $I$
gives rise to an isomorphism $I: \mathcal{H}(\partial A_C) 
\rightarrow \mathcal{H}(I(\partial A_C))$ is asserting that
there exists a symmetry in the four-dimensional theory
that has no natural origin in four-dimensions, as there
exists no four-dimensional diffeomorphism of $A_C$ that
when restricted to  $\partial A_C$ yields $I$. In other
words, it is, without any ``physical'' justification, enlarging 
the symmetry group of the theory. In point of fact, it is just
this enlarged symmetry group that we are seeing in
Theorem~\ref{Theorem:MainTheorem}.

The modifications that one must make to the TQFT axioms such
that they allow for detection of changes in smooth structure
are rather straightforward. To wit, one must limit the set of 
orientation-preserving diffeomorphisms that give rise to
isomorphisms of $\mathcal{H}(X)$. More specifically,
if $X$ is a closed, oriented 
$n$-dimensional smooth submanifold of a compact, oriented 
$(n+1)$-dimensional smooth manifold $W$, then any 
orientation-preserving diffeomorphism $f$ of $X$ that 
arises as a restriction of an orientation-preserving diffeomorphism
of $W$ induces an isomorphism $f: \mathcal{H}(X) \rightarrow 
\mathcal{H}(f(X))$. If $f'$ is an orientation-preserving 
diffeomorphism of $X$ that does not arise in such a manner,
then its action on $\mathcal{H}(X)$ is undefined. If we call
an orientation-preserving diffeomorphism $f$ that arises in such 
a manner a \textit{restricted} orientation-preserving diffeomorphism,
then the naturality and functoriality TQFT axioms must be modified
in the following manner so as to allow for detection of changes in 
smooth structure\footnote{One immediately sees that if one uses
these new axioms, the proof of Theorem~\ref{Theorem:MainTheorem}
fails.}.

%%%%%%%%%%%%%%%%%%%%%%%%%%%%%%%%%%%%%%%%%%%%%%%%%%%%%%%%%%%%%%%%%%%%%%%%
%%                                                                    %%
%%                             Naturality                             %%
%%                                                                    %%
%%%%%%%%%%%%%%%%%%%%%%%%%%%%%%%%%%%%%%%%%%%%%%%%%%%%%%%%%%%%%%%%%%%%%%%%
\subsection{Naturality}
\label{SubSection:Naturality}

\begin{axiom}[Naturality]
\label{Axiom:NaturalityII}
Any orientation-preserving diffeomorphism $f$ of $X$,
a closed, oriented $n$-dimensional smooth submanifold of $W$ a
compact, oriented $(n+1)$-dimensional smooth manifold,
that arises as a restriction of an orientation-preserving 
diffeomorphism of $W$
induces an isomorphism $f: \mathcal{H}(X) \rightarrow \mathcal{H}(f(X))$. 
For an orientation--preserving diffeomorphism $g$ from the cobordism 
$(W,X_-,X_+)$ to the cobordism $(W',X_-',X_+')$, the following diagram
is commutative.
\[
\xymatrixcolsep{5pc}
\xymatrix{
\mathcal{H}(X_-) \ar[d]_{Z(W)} \ar[r]^{g_{|_{X_-}}} &\mathcal{H}(X_-')\ar[d]^{Z(W')}\\
\mathcal{H}(X_+) \ar[r]^{g_{|_{X_+}}} &\mathcal{H}(X_+')}
\]
Note, $Z(W)$ is shorthand for $Z(W,X_-,X_+)$ and
$Z(W')$ is shorthand for $Z(W',X_-',X_+')$.
\end{axiom}

%%%%%%%%%%%%%%%%%%%%%%%%%%%%%%%%%%%%%%%%%%%%%%%%%%%%%%%%%%%%%%%%%%%%%%%%
%%                                                                    %%
%%                            Functoriality                           %%
%%                                                                    %%
%%%%%%%%%%%%%%%%%%%%%%%%%%%%%%%%%%%%%%%%%%%%%%%%%%%%%%%%%%%%%%%%%%%%%%%%
\subsection{Functoriality}
\label{SubSection:Functoriality}

\begin{axiom}[Functoriality]
\label{Axiom:FunctorialityII}
If a cobordism $(W,X_-,X_+)$ is obtained by gluing two cobordisms 
$(M,X_-,X)$ and $(M',X',X_+)$ using an orientation-preserving 
diffeomorphism $f$ where $f: X \rightarrow X'$ and $f$ can be viewed
as the restriction of an orientation-preserving diffeomorphism of $W$,
then following diagram is commutative.
\[
\xymatrixcolsep{5pc}
\xymatrix{
\mathcal{H}(X_-) \ar[d]_{Z(M)} \ar[r]^{Z(W)} &\mathcal{H}(X_+)\\
\mathcal{H}(X) \ar[r]^{f} &\mathcal{H}(X')\ar[u]_{Z(M')}}
\]
\end{axiom}

%%%%%%%%%%%%%%%%%%%%%%%%%%%%%%%%%%%%%%%%%%%%%%%%%%%%%%%%%%%%%%%%%%%%%%%%
%%%%%%%%%%%%%%%%%%%%%%%%%%%%%%%%%%%%%%%%%%%%%%%%%%%%%%%%%%%%%%%%%%%%%%%%
%%                                                                    %%
%%                             Conclusion                             %%
%%                                                                    %%
%%%%%%%%%%%%%%%%%%%%%%%%%%%%%%%%%%%%%%%%%%%%%%%%%%%%%%%%%%%%%%%%%%%%%%%%
%%%%%%%%%%%%%%%%%%%%%%%%%%%%%%%%%%%%%%%%%%%%%%%%%%%%%%%%%%%%%%%%%%%%%%%%
\section{Conclusion}
\label{Section:Conclusion}

%%%%%%%%%%%%%%%%%%%%%%%%%%%%%%%%%%%%%%%%%%%%%%%%%%%%%%%%%%%%%%%%%%%%%%%%
%%                                                                    %%
%%                               Remarks                              %%
%%                                                                    %%
%%%%%%%%%%%%%%%%%%%%%%%%%%%%%%%%%%%%%%%%%%%%%%%%%%%%%%%%%%%%%%%%%%%%%%%%
\subsection{Remarks}

The standard formulation of axiomatic TQFT~\cite{Blanchet2006232} is
sufficient for many situations in fewer than four dimensions. However,
in four-dimensions the standard axiomatic formulation requires some
small modifications if it is to detect changes in smooth structure. These
modifications are required as there exist orientation-preserving diffeomorphisms
of $\partial A_C$ that do not extend to orientation-preserving diffeomorphisms
of the smooth four-manifold $A_C$. (In fewer than four-dimensions such
diffeomorphisms do not exist\footnote{If they existed in fewer than 
four-dimensions, then there would exist exotic manifolds in three
or fewer dimensions. There exist no such manifolds.}.) If these 
small modifications are made, one obtains a set of axioms that allow
for the detection of changes in the smooth structure of a 
four-manifold\footnote{Note, $\mathcal{H}(X)$ may also have to
be infinite dimensional in four-dimensions.}.

%%%%%%%%%%%%%%%%%%%%%%%%%%%%%%%%%%%%%%%%%%%%%%%%%%%%%%%%%%%%%%%%%%%%%%%%
%%                                                                    %%
%%                           Axiomatic DQFT                           %%
%%                                                                    %%
%%%%%%%%%%%%%%%%%%%%%%%%%%%%%%%%%%%%%%%%%%%%%%%%%%%%%%%%%%%%%%%%%%%%%%%%
\subsection{Axiomatic DQFT}

We call the construct resulting from the modified axioms \textit{axiomatic differential quantum field theory}. In summary its axioms are as follows.

%%%%%%%%%%%%%%%%%%%%%%%%%%%%%%%%%%%%%%%%%%%%%%%%%%%%%%%%%%%%%%%%%%%%%%%%
%%                             Naturality                             %%
%%%%%%%%%%%%%%%%%%%%%%%%%%%%%%%%%%%%%%%%%%%%%%%%%%%%%%%%%%%%%%%%%%%%%%%%
\subsubsection{Naturality}

\begin{axiom}[Naturality]
Any orientation-preserving diffeomorphism $f$ of $X$,
a closed, oriented $n$-dimensional smooth submanifold of $W$ a
compact, oriented $(n+1)$-dimensional smooth manifold,
that arises as a restriction of an orientation-preserving 
diffeomorphism of $W$
induces an isomorphism $f: \mathcal{H}(X) \rightarrow \mathcal{H}(f(X))$. 
For an orientation--preserving diffeomorphism $g$ from the cobordism 
$(W,X_-,X_+)$ to the cobordism $(W',X_-',X_+')$, the following diagram
is commutative.
\[
\xymatrixcolsep{5pc}
\xymatrix{
\mathcal{H}(X_-) \ar[d]_{Z(W)} \ar[r]^{g_{|_{X_-}}} &\mathcal{H}(X_-')\ar[d]^{Z(W')}\\
\mathcal{H}(X_+) \ar[r]^{g_{|_{X_+}}} &\mathcal{H}(X_+')}
\]
Note, $Z(W)$ is shorthand for $Z(W,X_-,X_+)$ and
$Z(W')$ is shorthand for $Z(W',X_-',X_+')$.
\end{axiom}

%%%%%%%%%%%%%%%%%%%%%%%%%%%%%%%%%%%%%%%%%%%%%%%%%%%%%%%%%%%%%%%%%%%%%%%%
%%                            Functoriality                           %%
%%%%%%%%%%%%%%%%%%%%%%%%%%%%%%%%%%%%%%%%%%%%%%%%%%%%%%%%%%%%%%%%%%%%%%%%
\subsubsection{Functoriality}

\begin{axiom}[Functoriality]
If a cobordism $(W,X_-,X_+)$ is obtained by gluing two cobordisms 
$(M,X_-,X)$ and $(M',X',X_+)$ using an orientation-preserving 
diffeomorphism $f$ where $f: X \rightarrow X'$ and $f$ can be viewed
as the restriction of an orientation-preserving diffeomorphism of $W$,
then following diagram is commutative.
\[
\xymatrixcolsep{5pc}
\xymatrix{
\mathcal{H}(X_-) \ar[d]_{Z(M)} \ar[r]^{Z(W)} &\mathcal{H}(X_+)\\
\mathcal{H}(X) \ar[r]^{f} &\mathcal{H}(X')\ar[u]_{Z(M')}}
\]
\end{axiom}

%%%%%%%%%%%%%%%%%%%%%%%%%%%%%%%%%%%%%%%%%%%%%%%%%%%%%%%%%%%%%%%%%%%%%%%%
%%                            Normalization                           %%
%%%%%%%%%%%%%%%%%%%%%%%%%%%%%%%%%%%%%%%%%%%%%%%%%%%%%%%%%%%%%%%%%%%%%%%%
\subsubsection{Normalization}

\begin{axiom}[Normalization]
For any closed, oriented $n$-dimensional smooth manifold $X$, the
$\mathbb{F}$ linear map
\begin{equation*}
Z(X \times [0,1]): 
\mathcal{H}(X) \rightarrow \mathcal{H}(X)
\end{equation*}
is the identity.
\end{axiom}

%%%%%%%%%%%%%%%%%%%%%%%%%%%%%%%%%%%%%%%%%%%%%%%%%%%%%%%%%%%%%%%%%%%%%%%%
%%                          Multiplicativity                          %%
%%%%%%%%%%%%%%%%%%%%%%%%%%%%%%%%%%%%%%%%%%%%%%%%%%%%%%%%%%%%%%%%%%%%%%%%
\subsubsection{Multiplicativity}

\begin{axiom}[Multiplicativity]
There are functorial isomorphisms
\begin{equation*}
\mathcal{H}(X \amalg Y) \longrightarrow \mathcal{H}(X) \otimes \mathcal{H}(Y)
\end{equation*}
and
\begin{equation*}
\mathcal{H}(\emptyset) \longrightarrow \mathbb{F}
\end{equation*}
such that the diagrams
\[
\xymatrix{
\mathcal{H}((X_1 \amalg X_2) \amalg X_3) \ar[d] \ar[r] &(\mathcal{H}(X_1) \otimes \mathcal{H}(X_2)) \otimes \mathcal{H}(X_3)\ar[d]\\
\mathcal{H}(X_1 \amalg (X_2 \amalg X_3)) \ar[r] &\mathcal{H}(X_1) \otimes (\mathcal{H}(X_2) \otimes \mathcal{H}(X_3))}
\]
and
\[
\xymatrix{
\mathcal{H}(X \amalg \emptyset) \ar[d] \ar[r] &\mathcal{H}(X) \otimes \mathbb{F}\ar[d] \\
\mathcal{H}(X) \ar[r]^{id} &\mathcal{H}(X)}
\]
commute. Note, the vertical maps are induced by the obvious
diffeomorphisms and the standard vector space isomorphisms. 
\end{axiom}

%%%%%%%%%%%%%%%%%%%%%%%%%%%%%%%%%%%%%%%%%%%%%%%%%%%%%%%%%%%%%%%%%%%%%%%%
%%                              Symmetry                              %%
%%%%%%%%%%%%%%%%%%%%%%%%%%%%%%%%%%%%%%%%%%%%%%%%%%%%%%%%%%%%%%%%%%%%%%%%
\subsubsection{Symmetry}

\begin{axiom}[Symmetry]
The isomorphism
\begin{equation*}
\mathcal{H}(X \amalg Y) \longrightarrow \mathcal{H}(Y \amalg X)
\end{equation*}
induced by the obvious diffeomorphism corresponds to the standard
isomorphism of vector spaces
\begin{equation*}
\mathcal{H}(X) \otimes \mathcal{H}(Y) \longrightarrow \mathcal{H}(Y) \otimes \mathcal{H}(X).
\end{equation*}
\end{axiom}

%%%%%%%%%%%%%%%%%%%%%%%%%%%%%%%%%%%%%%%%%%%%%%%%%%%%%%%%%%%%%%%%%%%%%%%%
%%%%%%%%%%%%%%%%%%%%%%%%%%%%%%%%%%%%%%%%%%%%%%%%%%%%%%%%%%%%%%%%%%%%%%%%
%%                                                                    %%
%%                              Afterward                             %%
%%                                                                    %%
%%%%%%%%%%%%%%%%%%%%%%%%%%%%%%%%%%%%%%%%%%%%%%%%%%%%%%%%%%%%%%%%%%%%%%%%
%%%%%%%%%%%%%%%%%%%%%%%%%%%%%%%%%%%%%%%%%%%%%%%%%%%%%%%%%%%%%%%%%%%%%%%%
\section{Afterward}
\label{Section:Afterward}

Upon distributing this preprint, it has come to the author's attention
that a proof of a result similar to Theorem 4.1 was given as Theorem 4.1
of Freedman et al.~\cite{Freedman05}. In addition, the author was informed
of a research program with a focus similar to that of this preprint. This
research program was launched by Freedman, Kitaev, Nayak, Slingerland,
Walker, and Wang in~\cite{Freedman05}, continued by Kreck and Teichner 
in~\cite{Kreck08}, and furthered by Calegari, Freedman, and 
Walker in~\cite{Calegari10}.

%%%%%%%%%%%%%%%%%%%%%%%%%%%%%%%%%%%%%%%%%%%%%%%%%%%%%%%%%%%%%%%%%%%%%%%%
%%%%%%%%%%%%%%%%%%%%%%%%%%%%%%%%%%%%%%%%%%%%%%%%%%%%%%%%%%%%%%%%%%%%%%%%
%%                                                                    %%
%%                            Bibliography                            %%
%%                                                                    %%
%%%%%%%%%%%%%%%%%%%%%%%%%%%%%%%%%%%%%%%%%%%%%%%%%%%%%%%%%%%%%%%%%%%%%%%%
%%%%%%%%%%%%%%%%%%%%%%%%%%%%%%%%%%%%%%%%%%%%%%%%%%%%%%%%%%%%%%%%%%%%%%%%
\bibliographystyle{plain}
\bibliography{References}

%%%%%%%%%%%%%%%%%%%%%%%%%%%%%%%%%%%%%%%%%%%%%%%%%%%%%%%%%%%%%%%%%%%%%%%%
%%%%%%%%%%%%%%%%%%%%%%%%%%%%%%%%%%%%%%%%%%%%%%%%%%%%%%%%%%%%%%%%%%%%%%%%
%%                                                                    %%
%%                                End Document                        %%
%%                                                                    %%
%%%%%%%%%%%%%%%%%%%%%%%%%%%%%%%%%%%%%%%%%%%%%%%%%%%%%%%%%%%%%%%%%%%%%%%%
%%%%%%%%%%%%%%%%%%%%%%%%%%%%%%%%%%%%%%%%%%%%%%%%%%%%%%%%%%%%%%%%%%%%%%%%
\end{document}